\documentclass[11pt,letterpaper]{amsart}
\usepackage{graphicx} 
\usepackage{amsmath,amssymb,amsthm,amscd}\usepackage[usenames]{color}
\usepackage[all,cmtip]{xy}
\usepackage{enumerate}
\usepackage{cancel}
\usepackage{fullpage}
\xyoption{all}
\usepackage{xcolor}
\usepackage{tikz}
\usepackage{tikz-cd}
\usepackage{pgf,tikz}
\usetikzlibrary{arrows}
\usepackage[colorlinks=true, allcolors=blue]{hyperref}

\newtheorem{thm}{Theorem}[section]
\newtheorem{lem}[thm]{Lemma}

\newtheorem{prop}[thm]{Proposition}

\newtheorem{conjecture}{Conjecture}

\theoremstyle{definition}
\newtheorem{rem}[thm]{Remark}

\newcommand{\ZZ}{{\mathbb Z}}

\newcommand{\cA}{{\mathcal A}}

\newcommand{\cP}{{\mathcal P}}
\newcommand{\cR}{{\mathcal R}}

\newcommand{\tC}{{\widetilde C}}

\newcommand{\cH}{\mathcal H}

\newcommand{\ra}{\rightarrow}

\newcommand{\s}{{\sigma}}

\DeclareMathOperator{\Nm}{Nm}
\DeclareMathOperator{\Pic}{Pic}

\DeclareMathOperator{\Fix}{Fix}
\DeclareMathOperator{\Image}{Im}
\DeclareMathOperator{\Id}{Id}
\DeclareMathOperator{\Aut}{Aut}

\DeclareMathOperator{\Hom}{Hom}

\newcommand{\unnumberedfootnote}[1]{%
  \begingroup
  \renewcommand{\thefootnote}{}%
  \footnotetext[0]{#1}%
  \endgroup
}

\usepackage[textwidth=0.8in]{todonotes}
\title{
Finiteness and injectivity of Prym maps for cyclic coverings}

\author[P. Bor\'owka]{Pawe\l{} Bor\'owka}
\address{Pawe\l{} Bor\'owka \newline Institute of Mathematics, Jagiellonian University in Krak\'ow\\ ul. prof. Stanisława Łojasiewicza 6, 
30-348 Kraków, Poland}
\email{pawel.borowka@uj.edu.pl}

\author[J.C. Naranjo]{Juan Carlos Naranjo}
\address{Juan Carlos Naranjo \newline 1. Departament de Matem\`atiques i Inform\`atica, Universitat de Barcelona, Gran Via de les Corts Catalanes, 585, 08007 Barcelona, Spain \newline 2. Centre de Recerca Matemàtica, Edifici C, Campus Bellaterra, 08193 Bellaterra, Spain }
\email{jcnaranjo@ub.edu}

\author[A. Ortega]{Angela Ortega}
\address{Angela Ortega
\newline Institut für Mathematik, Humboldt Universität,
Unter den Linden 6, D-10099 Berlin, Germany}
\email{ortega@math.hu-berlin.de}

\author[A. Shatsila]{Anatoli Shatsila}
\address{Anatoli Shatsila \newline Doctoral School of Exact and Natural Sciences, Jagiellonian University in Krak\'ow \\ ul. prof. Stanisława Łojasiewicza 6, 
30-348 Kraków, Poland}
\email{anatoli.shatsila@doctoral.uj.edu.pl}

\begin{document}

\begin{abstract} 
The structure of the Prym map for coverings of degrees $d\geq 3$ is mostly unknown. Only recently, under mild numerical assumptions, a generic injectivity of the Prym maps for \'etale cyclic coverings of hyperelliptic curves of prime degrees has been shown. In the paper, we prove that the Prym maps is generically injective for all remaining degrees (i.e. composite numbers $d\geq 6$) and we prove global injectivity if $d$ is not a power of an odd prime. In particular, we complete the study of Prym maps of \'etale cyclic coverings of genus 2 curves.
As an application, we fully characterise for which $d$ the Prym map of cyclic coverings of degree $d$ of genus $3$ curves is generically finite and we conjecture that it is injective.

\end{abstract}

\maketitle

\unnumberedfootnote{\textit{2020 Mathematics Subject Classification:} 14H40, 14H30, 14H45, 14K12}
\unnumberedfootnote{\textit{Key words and phrases:} Prym variety, Prym map, coverings of curves.}

\section{Introduction}
Algebraic curves admitting automorphisms are objects of significant interest, particularly because they can be used to construct coverings. A special role is played by the \'etale cyclic coverings of curves of genus 2, as they appear as sections of (non-principal) line bundles on abelian surfaces.

Coverings of curves are also the main characters of the Prym theory, which relates them to polarised abelian varieties. More precisely, let $f:\tC\to H$ be a covering. Consider the norm map $\Nm_f:J\tC\to JH$ defined by $\Nm_f(\sum P_i)=\sum f(P_i)$, which is a well-defined group homomorphism. The connected 
component of its kernel containing the origin, denoted by $P(\tC/ H)$, is called the Prym variety of the covering. The Prym map associates the Prym variety $P(\tC/H)$ to a covering $f:\tC\to H$. Let $\mathcal{R}_g[d]$ be the moduli space parameterising cyclic 
coverings of degree $d$ over a smooth curve of genus $g$ and  $\mathcal{RH}_g[d]$ the locus in $\mathcal{R}_g[d]$ corresponding to the coverings whose base curve is hyperelliptic. We denote by $\mathcal{A}_g^{\delta}$
the moduli space of polarised abelian varieties of dimension $g$ with polarisation type $\delta$. In this paper we mainly deal with the  
Prym map $\mathcal{P}_g[d]: \mathcal{R}_g[d] \ra 
\mathcal{A}_{(g-1)(d-1)}^{\delta}$ (and its restriction to hyperelliptic locus $\mathcal{P}_{\cH_g}[d]: \mathcal{RH}_g[d] \ra 
\mathcal{A}_g^{\delta}$).  The type of the polarisation $\delta$ in the cyclic case is $\delta=(1,\ldots,1, d, \ldots, d)$ with $d$ appearing $g-1$ times and $1$  appearing  $(g-1)(d-2)$ times.

A central aim of the Prym theory is to understand the geometry of the fibres of this map.
Prym maps of \'etale double coverings of curves of genus $g\geq 7$ are generically injective, and examples of non-injectivity loci are known. Consequently, many established methods apply only generically. In particular, the results of \cite{Ago, LO16, NOS, NOPS} for degree $d\geq 3$ hold only generically,  although we
currently lack examples that demonstrate the necessity of this assumption. However, one should note that recent results have established the (global) injectivity of the Prym map in certain low degree cases, see \cite{Ikeda, BO20, NO22, BS26}. 

\subsection*{Results and methods}
The main objects of the paper are cyclic \'etale coverings of hyperelliptic curves of degree $d\geq 6$. The case $d$ prime  has previously been addressed in the literature , especially for curves of genus $2$. The $d$ prime hypothesis allows to simplify certain diagrams, and the number of abelian varieties involved is smaller. Nevertheless, the results obtained establish only generic injectivity (see \cite{NOS, NOPS}). In this paper, 
 we develop techniques to address the global injectivity problem for Prym maps for all degrees that are not powers of odd primes.
 In contrast with the degree prime situation, in most cases we can obtain global injectivity. 

Since the main results of the paper depend on the parity of the degree of the map, we present them in the following two theorems. 
\begin{thm}[Theorem \ref{maineven}]
For any $g\geq 2$ and $q\geq 3$ the Prym map
 $$\mathcal{P_H}_g[2q]: \mathcal{RH}_g[2q]\to\mathcal{A}_{(g-1)(2q-1)}^{(1,\ldots,
 2q, \ldots, 2q)}, \quad [\tC \stackrel{f}{\ra} H ] \mapsto [P(f), \Theta_P]$$ 
 is injective.
\end{thm}
For the case of odd non-prime degree, we have the following result.

\begin{thm}[Theorem \ref{mainodd} and Theorem \ref{mainodd2primes}]
Let $g \geq 2$ be an integer and $d\geq 9$ an odd non-prime number. The Prym map  $$\mathcal{P_H}_g[d]: \mathcal{RH}_g[d]\to\mathcal{A}_{(g-1)(d-1)}^{(1,\ldots,
 d, \ldots, d)}, \quad [\tC \stackrel{f}{\ra} H ] \mapsto [P(f), \Theta_P]$$ is generically injective if $d$ is a power of an odd prime and is injective if $d$ has at least two different prime divisors.
\end{thm} 

Our methods are explicit and we construct the inverse to the Prym map. We start by recovering the action of the dihedral group on a given Prym variety in the image of $\mathcal{P_H}_g[d]$ that is induced by the action on the top curve $\tC$, see Proposition \ref{group} and Remark \ref{Rem}. Then we apply isotypical decomposition and a careful investigation of quotient curves to find a map between quotient curves $C_0\to E_0$ that is not Galois and whose Galois closure is the top curve $\tC$. If the degree $d$ is not of the form $p^n$ for an odd prime $p$ we can uniquely extend the Galois group to the subgroup of automorphisms of $\tC$ containing a fixed-point free automorphism of order $d$ that yields the map $\tC\to H$. 

For $d=p^n$ with $p$ an odd prime, we can not  extend the Galois group intrinsically, so we need a very careful treatment to find an automorphism of order $p^n$. In particular, we need a generality assumption and we use the results of \cite{NOPS}, so we can only prove generic injectivity of the Prym map. Note that if $n\neq 1$ we do not need numerical conditions on $p$ that are needed in \cite[Theorem 1.2]{NOPS}. 

It is known that the Prym map for cyclic coverings of degree $d\geq 2$ of genus $g\geq 7$ curves is generically finite, see \cite{LOdedicata}. However, for $g=3,4,5,6$ there have been no results on generic finiteness, except for $g=4,d=3$, see \cite{Faber} and $g=6,d=2$, see \cite{DS}. As an application of our results, we obtain the characterisation of generic finiteness for $g=3$ in the following theorem.
\begin{thm}[Theorem \ref{finite}]
    The Prym map of \'etale cyclic coverings of degree $d$ of genus $3$ curves 
    $$\cP_3[d]:\cR_3[d]\to\cA_{(2d-2)}^{(1,\ldots 1,d,d)}$$
    is generically finite if and only if $d\geq4$.
    \end{thm}

Since our results  explicitly compute the degree of the Prym map for the locus of hyperelliptic curves instead of investigating the injectivity of the codifferential, we conjecture that the generic degree of the Prym map of cyclic coverings of degrees $d\geq 4$ actually equals one, see Conjecture \ref{conjecture}.

\subsection*{Structure.} The structure of the paper is as follows. In Section 2, we study the automorphisms of the Prym variety $P$ that fix each element in the kernel of the natural polarisation $\Theta_P$. This is the key result for reconstructing the covering $\widetilde{C} \to H$ from the polarised abelian variety $(P(\widetilde{C}/H), \Theta_P)$. We also establish some auxiliary results concerning pullbacks of coverings that naturally arise in our setting. Section 3 is devoted to the study of the Prym map for cyclic coverings of even degree. In Section 4, we focus on cyclic coverings of odd degree. Finally, Section 5 is devoted to the investigation of the Prym maps of genus $3$ curves and we state some remarks and open questions.

\subsection*{Acknowledgements.} 
P. Bor\'owka has been supported by the Polish National Science Centre project number 2024/54/E/ST1/00330. 
J.C. Naranjo was partially supported by the Spanish MINECO research project PID2023-147642NB-I00. A. Shatsila was supported by the Polish National Science Centre project number 2024/53/N/ST1/01634 and by ID.UJ Research Support Module program at the Faculty of Mathematics and Computer Science of Jagiellonian University in Krak\'ow.

\bigskip

\section{Preliminary results}
Recall that $\mathcal{P_H}_g[d]:\mathcal{RH}_g[d] \ra 
\mathcal{A}_g^{\delta}$ is the Prym map of \'etale cyclic coverings of degree $d$ over hyperelliptic genus $g$ curves, where $g\geq 2$. Let $f:\tC\to H$ be such a covering. Let $\sigma\in\Aut(\tC)$ be a deck transformation that induces the covering $f$ and $j \in\Aut(\tC)$ be a lift of the hyperelliptic involution $\iota$ from $H$. The following proposition is an analogue  of \cite[Proposition 3.1]{NOS} and \cite[Proposition 3.1]{BS26}. 

\begin{prop}
\label{group}
    Let $d \geq 3$ be an integer and $(P, \Theta_P)$ be an element of $\Image\mathcal{P_H}_g[d]$. 
    Then the subgroup of automorphisms $$G := \{ \phi \in \Aut(P,\Theta_P) \: | \: \phi(x) = x \:\: \forall x \in K(\Theta_P)\}.$$
    is isomorphic to $\langle \sigma, -j \rangle \simeq {D}_{d}$.
\end{prop}

\begin{proof} Let $f: \tC \to H$ be a covering with ${P}(f) = P$ and Galois group $\langle \sigma \rangle \simeq \mathbb{Z}_{d}$. Let $\sigma_{|P},\ j_{|P}$ be the automorphisms of $(P,\Theta_P)$ induced  by $\sigma$, respectively by a lift of the hyperelliptic involution, on $J\tC$. 
Note that $K(\Theta_P) = f^*JH \cap P(f) \subset \Fix(-j) \cap \Fix(\sigma)$ since $\sigma_{|f^*JH} = \text{id}_{JH} = (-j)_{|f^*JH}$. Therefore, we have $D_{d} \simeq \langle \sigma, -j \rangle \subseteq G$.  

Let $\psi \in G$. Then, there is an automorphism 
$\Tilde{\psi}: J\tC \to J\tC$ such that the following diagram commutes:


\begin{equation}\label{extended_action}
        \begin{tikzcd}
0 \arrow[r] & K(\Theta_P) \arrow[d, equal] \arrow[r] & f^*JH \times P \arrow[d, "{(\text{id},\psi)}"] \arrow[r, "\mu"] & J\tC \arrow[d, "\Tilde{\psi}"] \arrow[r] & 0 \\
0 \arrow[r] & K(\Theta_P) \arrow[r]                                & f^*JH \times P \arrow[r, "\mu"]                           & J\tC \arrow[r]                           & 0
\end{tikzcd},
\end{equation}
where $\mu$ is the addition map. We shall show that $\Tilde{\psi}$ is a polarised isogeny. It follows from Diagram \eqref{extended_action} that  
$\mu^*\Tilde{\psi}^*\mathcal{O}_{J\tC}(\Tilde{\Theta})$ and $ \mu^*\mathcal{O}_{J\tC}(\Tilde{\Theta})$ are equal as polarisations on $f^*JH \times P$ because $(\text{id},\psi)$ is a polarised isomorphism.
Now, since $\mu^*$ has a finite kernel,  
$\Tilde{\psi}^*\mathcal{O}_{J\tC}(\Tilde{\Theta}) \otimes \mathcal{O}_{J\tC}(\Tilde{\Theta})^{-1}$ is a torsion sheaf, hence it belongs to $\Pic^0(J\tC)$. 
Therefore, $\Tilde{\psi}^*\mathcal{O}_{J\tC}(\Tilde{\Theta})$ induces the canonical polarisation $\Tilde{\Theta}$ on $J\tC$ and thus $\tilde \psi \in \Aut(J\tC, \tilde{\Theta})$. 

Recall that \cite[Exercise 11.19]{BL} $$\Aut(\tC) = \begin{cases} \Aut(J\tC, \tilde{\Theta})  \text{ if } \tC \text{ is hyperelliptic} \\  \Aut(J\tC, \tilde{\Theta})/ \langle -1_{J\tC} \rangle \text{ if } \tC \text{ is nonhyperelliptic.} \end{cases}$$ Therefore,  there is an automorphism $\Tilde{\psi}_0$ of $\tC$ inducing either $\Tilde{\psi}$ or $-\Tilde{\psi}$, in other words, inducing $(-1)^i\Tilde{\psi}$, where $i \in \{0,1\}$. We will prove that if $i=0$ then $\Tilde{\psi}_0 \in \text{Gal}(\tC/H) = \langle \sigma \rangle $ and if $i = 1$ then $\Tilde{\psi}_0 \in \{j, j\sigma, \ldots, j\sigma^{d-1}\}$, that is,  $\Tilde{\psi}_0$ is either a lift of the identity or a lift of the hyperelliptic involution in $H$. 

Let us first show that $\Nm_f \circ (Id_{J\tC} - \Tilde{\psi}) = 0$. Writing $x \in J\tC$ as $y + f^*(z)$ with $y \in P, z \in JH$ we have $$(\Id_{J\tC} - \Tilde{\psi})(x) = y + f^*(z) - (\Tilde{\psi}(y) + f^*(z)) = y - \Tilde{\psi}(y).$$ Since $\Tilde{\psi}(y) \in P$ for any $y \in P$ we get $\Nm_f \circ (Id_{J\tC} - \Tilde{\psi}) = 0$. Let $p_1,p_2 \in \tC$ be two points in the same fibre  of $f$. Then
$\Nm_f(p_1 - p_2 -(-1)^i( \Tilde{\psi}_0(p_1) - \Tilde{\psi}_0(p_2))) = 0$, hence $f(\Tilde{\psi}_0(p_1)) = f(\Tilde{\psi}_0(p_2))$. Therefore, $\Tilde{\psi}_0$ descends to an automorphism $\psi_0$ on $H$. Denoting the corresponding automorphism of $JH$ by the same letter we have $$f^* \circ \psi_0 = (-1)^i\Tilde{\psi}_0 \circ f^* = (-1)^if^*$$ implying $$f^*\circ ((-1)^i_{JH} - \psi_0) = 0.$$ Since $\ker f^*$ is finite, we get $\psi_0 = (-1)^i_{JH}$, hence $\psi_0 \in \Aut(H)$ is either the identity or the hyperelliptic involution. Therefore, $\Tilde{\psi}_0$ is either a lift of the identity or a lift of the hyperelliptic involution in $H$, hence it is enough to show that $\psi \mapsto \Tilde{\psi}_0$ is an injection. Given two elements $\varphi, \psi \in G$ inducing the same $\alpha \in \Aut(\tC)$ we see that either $\tilde{\varphi} = \tilde{\psi}$ or $\tilde{\varphi} =  (-1)_{J\tC}\circ \tilde{\psi}$. The latter case is impossible since $\tilde{\varphi}_{|f^*JH} = \tilde{\psi}_{|f^*JH} = \text{id}_{f^*JH}$, hence $\tilde{\varphi} = \tilde{\psi}$ implying $\varphi = \psi$. 
\end{proof} 

\begin{rem}\label{Rem}
    A similar statement in \cite[Proposition 3.1]{NOS} is incorrect. The subgroup of automorphisms fixing the kernel of 
    the polarisation should be in that case $D_{d}$ instead of $\ZZ_d$ (where $d$ is a prime number) and contains $-j$ where $j$ is a lift of hyperelliptic involution. The mistake has no impact on the rest of the paper \cite{NOS}, as well as on \cite{NOPS} where it is also applied, since $\ZZ_d$ is a distinguished subgroup in $D_d$.  
\end{rem}

 We will use the following notation for the decomposition of the Prym variety. Let $A$ be an abelian variety and $B_1,\ldots,B_k$ its abelian subvarieties. Let $\varepsilon_i$ be the associated symmetric idempotent of $B_i$, see \cite[\S 5]{BL}. We write $A=B_1\boxplus\ldots\boxplus B_k$ if $\varepsilon_1+\ldots+\varepsilon_k=1$, i.e. if the $B_i$'s are complementary abelian subvarieties (see \cite[Definition 5.1]{BO19}). We call $B_1,B_2$ orthogonal if $\varepsilon_{B_1}+\varepsilon_{B_2}$ is an idempotent. Note that being orthogonal (or complementary) depends on the polarisation of the ambient abelian variety.

We need to describe some quotient curves of the top curve of the covering $f:\tC\to H$. Let $p$ be a prime divisor of $d$ and $q=\frac{d}{p}$, with $q\geq 2$. As before, let $\sigma$ be a deck transformation that defines $f$ and let $j$ be a lift of the hyperelliptic involution on $H$. We denote by $C_{\sigma^p}=\tC/\langle\sigma^p\rangle,\ C_0=\tC/\langle j\rangle, E_0=\tC/\left<\sigma^p, j\right>$ We have the following commutative diagram:
\begin{equation}
\label{towerforn}
\begin{tikzcd}
\tC \arrow[rrdd, "q:1"] \arrow[rrrd, "2:1"] \arrow[ddd, "d:1"] &  &                                                    &                        &                        \\
                                                                &  &                                                    & C_0 \arrow[rdd, "q:1"] &                        \\
                                                                &  & C_{\sigma^p} \arrow[lld, "p:1"] \arrow[rrd, "2:1"] &                        &                        \\
H \arrow[rrd, "2:1"]                                            &  &                                                    &                        & E_0 \arrow[lld, "p:1"] \\
                                                                &  & \mathbb{P}^1                                       &                        &                       
\end{tikzcd}
\end{equation}

From the Diagram \eqref{towerforn} and from the definition of Prym variety, we can deduce the following decomposition of $J\tC$:
\begin{equation}
J\tC=f^*JH\boxplus P(\tC/H)=f^*JH\boxplus P(C_{\sigma^p}/H)\boxplus P(\tC/C_{\sigma^p}).   
\end{equation}
Later on, we will refine this decomposition in more detailed cases. Now, we would like to show that $JE_0$ is embedded in $J\tC$.

\begin{lem}\label{embeddingE}
Let $E_0 := \tC/\langle\sigma^p,j\rangle$ and $\pi: \tC \to E_0$ be the quotient map. Then $\pi^*:JE_0\to J\tC$ is an embedding. In particular, $\pi:\tC\to E_0$ does not factorise via $\tC\to C''\to E_0$ where $C''\to E_0$ is cyclic \'etale. 
\end{lem}
\begin{proof}
Denote by $g:\tC\to C_{\sigma^p}$ and $h:C_{\sigma^p} \to H$ the quotient maps. Then, by \cite[Proposition 2.4]{Ortega}, we have $JC_{\sigma^p}=h^*JH\boxplus P(h)$, and $JE_0$ embeds into $P(h)\subseteq JC_{\sigma^p}$. By \cite[Corollary 12.1.4]{BL} and \cite[Proposition 3.2.9]{LR} we get $P(h)\cap h^*JH=K(h^*JH)\cong \ker(h^*)^{\perp}/\ker(h^*)$.
Note that, by construction, $\ker(h^*)^\perp\subseteq JH[p]$ because $h$ is cyclic of degree $p$.
On the other hand, if $\eta\in JH[d]$ is a line bundle defining the covering $f$ of degree $d$, then $\eta^q$ is a line bundle defining $h$ and $\left<\eta\right>\cap JH[p]=\left<\eta^q\right>=\ker (h^*)$. Since $\left<h^*(\eta)\right>=\ker(g^*)$, we get $0=\ker(g^*)\cap K(h^*JH)=\ker(g^*)\cap P(h)$. This shows that $g^*_{|P(h)}$ is an embedding, so $JE_0$ is embedded in $J\tC$.

The second part of the statement follows directly from \cite[Proposition 11.4.3]{BL}.
\end{proof}

We will also need the following lemma.
\begin{lem}\label{key_lemma}
    Let $h:C'\to C$ be a covering that does not factorise via an \'etale covering $C'\to C''\to C$.
        Then $h^*: JC\to JC'$ is an embedding and any other embedding $i:JC\to JC'$ with $\Image i=\Image h^*$ is (up to an isomorphism of $JC$) given by $h^*$.
\end{lem}
\begin{proof} 
    Let $h:C'\to C$ be a covering. Then 
    $\Nm_h:JC'\to JC$ given by $\Nm_h(\sum P)=\sum h(P)$ is the unique extension of $h$ to Jacobians and by \cite[Proposition 11.4.3]{BL} the pullback map $h^*:JC\to JC'$ is an embedding. Moreover, by \cite[Equation 2 p. 332]{BL}) $h^*$ is dual to $\Nm_h$ after identifying Jacobians with their duals.

    If $i$ is another embedding then certainly $\alpha=i^{-1}\circ h^*$ is an automorphism of $JC$ and $h^*=i\circ \alpha$. Note that if $C$ is an elliptic curve we need to abuse notation and treat a translation on an elliptic curve as an automorphism.
\end{proof}

\section{Cyclic coverings of even degree}\label{DCov}

Let $d=2q \geq 6$ and $(\tC, H) \in \cR\cH_g[2q]$ be a degree $2q$ cyclic covering of a genus $g$ hyperelliptic curve $H$. We denote by $\sigma$ an automorphism of order $2q$ on $\tC$, such that $H =\tC/\langle \sigma \rangle$. The hyperelliptic involution on $H$ lifts to an involution $j$ on $\tC$, in such a way that the dihedral group 
$$
D_{2q}=\langle \sigma, j \ : \ \sigma^{2q}=j^2=1,\ \sigma j = j\sigma^{-1}   \rangle
$$
acts as automorphisms on $\tC$. 
This action gives rise to the following tower of curves:
\begin{equation}
    \label{big-tower}
\begin{tikzcd}
                      &                          &                                                                      & \tC \arrow[ld, "2:1"] \arrow[ddd, "2q:1"] \arrow[rd, "q:1"'] \arrow[llld, "2:1"] \arrow[rrrd, "2:1"] &                                                                     &                         &                        \\
C_0 \arrow[rd, "2:1"] &                          & C_{\sigma^q} \arrow[ld, "2:1"'] \arrow[rdd, "q:1"] \arrow[d, "2:1"'] &                                                                                                      & C_{\sigma^2} \arrow[rd, "2:1"] \arrow[ldd, "2:1"'] \arrow[d, "2:1"] &                         & C_1 \arrow[ld, "q:1"'] \\
                      & F_0 \arrow[rrdd, "q:1"'] & F_1 \arrow[rdd, "q:1"']                                              &                                                                                                      & E_0 \arrow[ldd, "2:1"]                                              & E_1 \arrow[lldd, "2:1"] &                        \\
                      &                          &                                                                      & H \arrow[d, "2:1" description]                                                                       &                                                                     &                         &                        \\
                      &                          &                                                                      & \mathbb{P}^1                                                                                         &                                                                     &                         &                       
\end{tikzcd}
\end{equation}

Here, $C_{\sigma^2} := \tC /\langle \sigma^2 \rangle$, $C_{\sigma^q} := \tC /\langle \sigma^q \rangle$, $C_i = \tC / \langle j\sigma^i \rangle$, $E_0 = \tC/\langle \sigma^2, j\rangle, E_1 = \tC/ \langle j\sigma, \sigma^2  \rangle$, $F_0 = \tC / \langle j, 
\sigma^q\rangle, F_1 = \tC / \langle j\sigma, 
\sigma^q \rangle$. Note that if $q$ is odd, the curves $F_0$ and $F_1$ are isomorphic.
By the Riemann-Hurwitz formula, genera of $\tC$, $C_{\sigma^q}$ and $C_{\sigma^2}$ are $2q(g-1)+1$, $q(g-1)+1$ and $2g-1$ respectively. We denote by $I_{even} = \{j, j\sigma^2, \ldots, j\sigma^{2q-2}\}$ and $I_{odd} = \{j\sigma, j\sigma^3, \ldots, j\sigma^{2q-1}\}$ two disjoint conjugacy classes of involutions in $D_{2q}$. In particular,  the curves $C_0$, $C_2, \ldots, C_{2q-2}$ are isomorphic to each other and the same for $C_1$, $C_3, \ldots, C_{2q-1}$.

\begin{lem}\label{genera}
There exists some \(1\leq m \leq \lfloor  \frac{g+1}{2}\rfloor\) such that $g(C_0) = q(g-1)-g+m$, $g(C_1) = q(g-1)-m+1$, $g(E_1) = g-m$, $g(E_0) = m-1$. Moreover, $g(F_0) = \frac{q(g-1)}{2}+m-g, g(F_1)=\frac{q(g-1)}{2}+1-m$ if $q$ is even and $g(F_0) = g(F_1) =  \frac{(q-1)(g-1)}{2}$ if $q$ is odd.     
\end{lem}

\begin{proof}
     Let $\eta$ be a line bundle on $H$ defining the covering $f: \tC \to H$. Then $\eta^q \in \Pic^0(H)[2]$ so $\eta^q  = \mathcal{O}_H(\sum_{i=1}^m(w_{2i}-w_{2i-1}))$,
     where $w_1, \ldots, w_{2m}$ are Weierstrass points on $H$ for some \(1\leq m \leq \lfloor  \frac{g+1}{2}\rfloor\). The hyperelliptic involution $\iota$ lifts to the hyperelliptic involution $j$ on $C_{\sigma^2}$ and another involution $j\sigma$. We choose these lifts
     so that the fixed points of $j\sigma$ are the fibers of $w_1,\ldots, w_{2m}$ and the fixed points of $j$ are the fibers of the remaining $2g+2-2m$ Weierstrass points. Since the involutions in $I_{odd}$ are the liftings of $j\sigma$ under the covering map $\tC \to C_{\sigma^2}$, 
     the union of all the fixed points of these involutions are precisely the fibers $f^{-1}(w_1), \ldots, f^{-1}(w_{2m})$. Therefore, every involution in $I_{odd}$ has $4m$ fixed points and every 
     involution in $I_{even}$ has $4g+4-4m$ fixed points. 
     From the Riemann-Hurwitz formula, it follows that $g(C_0) = q(g-1)-g+m$, $g(C_1) = q(g-1)-m+1$, $g(E_1) = g-m$, and $g(E_0) = m-1$.

The involution $j$ on $C_{\sigma^q}$ has $2g+2-2m+(2g+2-2m) = 4g+4-4m$ fixed points if $q$ is even and $2g+2-2m+2m = 2g+2$ if $q$ is odd, hence $g(F_0) = \frac{q(g-1)}{2}+m-g$ if $q$ is even and $g(F_0) = \frac{(q-1)(g-1)}{2}$ if $q$ is odd. Analogously, $g(F_1) = \frac{q(g-1)}{2}+1-m$ if $q$ is even and $g(F_1) = \frac{(q-1)(g-1)}{2}$ if $q$ is odd. 
\end{proof}

\begin{rem}
    Lemma \ref{genera} shows that the moduli $\mathcal{RH}_g[2q]$ is not irreducible for $g\geq 3$. Since the number of Weierstrass points in the presentation of a $2$-torsion point being $2m$ is unique (up to a complement $2g+2-2m$), we have at least $\lfloor\frac{g+1}{2}\rfloor$ irreducible components of $\mathcal{RH}_g[2q]$.
    \end{rem}

We also note that $C_{\sigma^2}$ is a hyperelliptic curve if and only if $m=1$ by \cite[Proposition 4.2]{BO19}. We would like to focus on the hyperelliptic curve $E_1$ of genus $g-m \geq 1$. 

We will need to following lemma.
\begin{lem}\label{elliptic}
   Assume that we are in the situation of Diagram \eqref{big-tower}.
   Then, the covering $C_1 \to E_1$ is never Galois. Its Galois closure is $\tC\to E_1$. Moreover, the hyperelliptic involution of $E_1$ (or $(-1)_{E_1}$ if $g-m=1$) lifts to automophisms $j, j\sigma^2,\ldots, j\sigma^{2q-2}, \sigma, \sigma^3, \ldots, \sigma^{2q-1}$ of $\tC$.
\end{lem}
\begin{proof}
    
    Recall that $C_1 = \tC / \langle j\sigma^1 \rangle$ and $E_1 = \tC/ \langle j\sigma, \sigma^2  \rangle$. Since  
    $\sigma^2j\sigma=j\sigma^{2q-1}\neq j\sigma^3$, we have that $\sigma^2$ is not an automorphism of $C_1$ and hence $C_1\to E_1$ is not Galois. Since the covering $\tC\to E_1$ is of degree $2q$ and it is Galois, by minimality it has to be a Galois closure of $C_1\to E$.  

    Now, consider a commutative subdiagram
    \begin{equation*}
\xymatrix{
&\tC \ar[dl]_{q:1} \ar[dr]^{2:1}& \\
C_{\sigma^2} \ar[dr]_{2:1} & & C_1 \ar[dl]^{q:1} \\
& E_1 &
}
\end{equation*}
The hyperelliptic involution $\iota_{E_1}$ (or $(-1)_{E_1}$ if $g-m=1$) is induced by the involution $j$ of $C_{\sigma^2}$ and the lifts of $\iota_{E_1}$ to $C_{\sigma^2}$ are $j$ and $\sigma$. Therefore, $\iota_{E_1}$ lifts to automophisms $j, j\sigma^2,\ldots, j\sigma^{2q-2}, \sigma, \sigma^3, \ldots, \sigma^{2q-1}$ of $\tC$.
\end{proof}

For any map between curves $g:C'\to C$, if the pullback map  $g^*:JC\to JC'$ is injective then we consider $JC$ as a subvariety of $JC'$. If the pullback map is not injective, then we denote the image by $g^*(JC)$.
Now, we are ready to derive some information from Diagram \eqref{big-tower}.

\begin{prop}
\label{decomp}
    Using the notation from Diagram \eqref{big-tower} we have 
    \begin{enumerate}
        \item $P(C_{\sigma^2}/H)$ is isomorphic to $JE_0 \times JE_1$. If $g=2$, then $E_0 \simeq \mathbb{P}^1$, $E_1$ is an elliptic curve and $P(C_{\sigma^2}/H) = E_1$.
        \item If \(4 \nmid q\) then $P(C_{\sigma^q}/H)$ is isomorphic to $JF_0 \times JF_1$. If $4 \mid q$ then $P(C_{\sigma^q}/H)$ is isogenous to $JF_0 \times JF_1$;
        \item Let $m = 1$ (in particular, $C_{\sigma^2}$ is hyperelliptic). If \(4 \nmid q\) then $P(\tC/C_{\sigma^2})$ is isomorphic to  $JC_0 \times JC_0$. If \(4 \mid q\) then $P(\tC/C_{\sigma^2})$ is isogenous to  $JC_0 \times JC_0$;
        \item Let $G = P(C_0/F_0)$ and $m=1$. Then $P(\tC/H)$ is  isogenous to $G^2\times JF_0^2 \times JE_0 \times JE_1$. 
        \item Let $q=3$ and $g=2$. Then $G$ is an elliptic curve equal to the quotient of $C_0$ by the involution $\iota_0\sigma^3$, where $\iota_0$ is the hyperelliptic involution on $C_0$. In particular, $P(\tC/H)$ is completely decomposable.  
    \end{enumerate}
    
\end{prop}

\begin{proof}
    Part $(i)$ follows from \cite{Ortega} and Lemma \ref{genera}. Part $(ii)$ is a consequence of \cite{Ortega} when $4 \nmid q$ and \cite{LO18} when $4 \mid q$.
    Part $(iii)$ is analogous to $(ii)$ after noticing that in the case $m=1$ the curve $C_{\s^2}$ is hyperelliptic.

    To prove $(iv)$ notice that $$P(\tC/H) \sim P(\tC/C_{\sigma^2}) \times JE_0 \times JE_1 \sim JC_0^2 \times JE_0 \times JE_1 \sim G^2 \times JF_0^2 \times JE_0 \times JE_1,$$ where we write $A \sim B$ if $A$ is isogenous to $B$.

    For $q=3$ and $g=2$, $C_0$ is a genus 2 curve with the hyperelliptic involution $\iota_0$, hence $P(C_0/F_0) = G$, where $G = C_0/\langle \iota_0 \sigma^3 \rangle$
    is an elliptic curve.

\end{proof}

Recall that $\mathcal{RH}_g[2q]$ denote the moduli space of degree $2q$ cyclic \'etale coverings of  hyperelliptic curves of
genus $g$. We consider the Prym map $$\mathcal{P}_{\mathcal{H}_g}[2q]: \mathcal{R}\mathcal{H}_g[2q]\to\mathcal{A}_{(2q-1)(g-1)}^{(1,1,\ldots,1, 2q,\ldots,2q)}.$$ 
\begin{thm}\label{maineven}
Let $g\geq 2$ and $q\geq 3$. Then, the Prym map  $\mathcal{P_H}_g[2q]: \mathcal{RH}_g[2q]\to\mathcal{A}_{(2q-1)(g-1)}^{(1,1,\ldots,1, 2q, \ldots, 2q)}$ is injective.
\end{thm} 

\begin{proof}
	Let $(P,\Theta_P)$ be an element of the image of $\mathcal{P_H}_g[2q]$. By Proposition \ref{group} we recover $\langle \sigma, -j  \rangle \in \Aut(P,\Theta_P)$ as a subgroup isomorphic to $D_{2q}$. 
    Note that $\mathbb{Z}_{2q} \simeq \langle \sigma \rangle \subsetneq \langle \sigma, -j \rangle \simeq D_{2q}$ is the unique subgroup of 
    $D_{2q}$ isomorphic to $\mathbb{Z}_{2q}$, hence we can choose $\sigma \in \Aut(P,\Theta_P)$ to be one of the generators of $\mathbb{Z}_{2q}$. 
    The elements of $D_{2q} \setminus \mathbb{Z}_{2q}$ are involutions of the form $-j\sigma^s$, where $s = 0,\ldots, 2q-1$. There are two 
    conjugacy classes of them, namely $I_{odd} = \{-j\sigma, -j\sigma^3,\ldots, -j\sigma^{2q-1}\}$ and $I_{even} = \{-j, -j\sigma^2,\ldots, -j\sigma^{2q-2}\}$. It follows from Lemma 
    \ref{genera} that there exists $m \in \{1,2,\ldots, \lfloor \frac{g+1}{2} \rfloor\}$ such that $\Image(1+j\sigma^s) = JC_s \subset P$ is of dimension $q(g-1)-g+m$ if $s$ is even and of dimension $q(g-1)-m+1$ if $s$ is odd. First, we consider the case when \(m \neq \frac{g+1}{2}\). Note that $q(g-1)-g+m < q(g-1)-m+1$ for any $m < \frac{g+1}{2}$, so we can distinguish the conjugacy classes by dimension. Take any involution in $I_{odd}$, which we will denote by $j\sigma$. 
    Note that $\Image(1+j\sigma)(1+\sigma^2+\ldots+\sigma^{2q-2}) = JE_1 \subsetneq JC_1$, so there is a natural inclusion map $i: JE_1 \to JC_1$. According to Lemma \ref{key_lemma}, this inclusion corresponds to a unique covering 
     $C_1 \to E_1$, which is precisely the covering map of degree $q$ given in Diagram \eqref{big-tower}. 

By Lemma \ref{elliptic}, the Galois closure of $C_1 \to E_1$ provides the covering 
$\tC \to E_1$ together with an automorphism $\sigma^2$ of order $q$ and an involution $j\sigma$ on $\tC$. Moreover, by lifting the hyperelliptic involution $\iota_{E_1}$ on \(E_1\) to \(\Tilde{C}\) we recover the set $I = \{j, j\sigma^2,\ldots, j\sigma^{2q-2}, \sigma, \sigma^3, \ldots, \sigma^{2q-1}\}$. Any element of order \(2q\) in $I$ generates the group $\mathbb{Z}_{2q} \simeq \langle \sigma \rangle$, hence we recover the covering $f: \tC \to H$. 

If \(m = \frac{g+1}{2}\) then we can take any involution in \(D_{2q}\setminus \ZZ_{2q}\). Indeed, in the proof of Lemma \ref{genera} we made a choice of the lifts so that all involutions in $I_{odd}$ have $4m$ fixed points and all involutions in $I_{even}$ have $4g+4-4m$ fixed points. In this case the number of fixed points is equal to $2g+2$ for any involution in $I_{odd} \cup I_{even}$, so all claims for involutions in $I_{odd}$ are also valid for $I_{even}$. Therefore, if we take the involution \(j \in I_{even}\), then, repeating the argument (together with Lemma \ref{elliptic}), we see that we can recover the curve \(\tC\) together with the set \(I = \{j\s, j\sigma^3,\ldots, j\sigma^{2q-1}, \sigma, \sigma^3, \ldots, \sigma^{2q-1}\}\), so, in particular, the covering $f: \tC \to H$.

\end{proof}
\begin{rem}
    In the proof of Theorem \ref{maineven} we have shown that $JE_1$ can be intrinsically recovered from $(P,\Theta_P)$. The dimension
    $\dim E_1 = g-m$ reflects the fact that the image of the Prym map also has at least $\lfloor \frac{g+1}{2}\rfloor$ components.
\end{rem}

\section{Cyclic coverings of odd degree}\label{sectodd}

In this section, we show that for an  odd non-prime integer \(d \geq 9\) the Prym map $\mathcal{P_H}_g[d]$ of degree \(d\) cyclic \'etale coverings of genus $g$ hyperelliptic curves $$\mathcal{P_H}_g[d]: \mathcal{RH}_g[d]\to\mathcal{A}_{(g-1)(d-1)}^{(1,1,\ldots, 1,d,\dots,d)}$$ is generically injective if $d$ is a power of a prime number and injective for all the other odd numbers. The case $d$ prime
was considered in \cite{NOPS}.
Let us assume $d=pq$ for $p$ a prime number dividing $d$ and $q=\frac{d}{p}\geq 3$ ($q$ is not necessarily prime). 

Let $f: \tC \to H$ be an \'etale cyclic covering of degree $d$ with $H$ a hyperelliptic curve of genus $g$ and $\sigma \in \Aut(\tC)$ be an automorphism inducing $f$. We denote the lifts of the hyperelliptic involution $\iota$ on $H$ by $j, j\sigma, \ldots, j\sigma^{pq-1}$. Note that since $d$ is odd, all involutions are conjugate, hence all quotients of $\tC$ by these involutions are isomorphic. 
Define the curves $C_{\sigma^p} := \tC/\langle \sigma^p \rangle, C_0 = \tC/\langle j \rangle, E_0 = \tC / \langle j, \sigma^p\rangle$, which fit in  following commutative diagram 

\begin{equation}
\label{small-tower}
\begin{tikzcd}
\tC \arrow[rrdd, "q:1"] \arrow[rrrd, "2:1"] \arrow[ddd, "d:1"] &  &                                                    &                        &                        \\
                                                                &  &                                                    & C_0 \arrow[rdd, "q:1"] &                        \\
                                                                &  & C_{\sigma^p} \arrow[lld, "p:1"] \arrow[rrd, "2:1"] &                        &                        \\
H \arrow[rrd, "2:1"]                                            &  &                                                    &                        & E_0 \arrow[lld, "p:1"] \\
                                                                &  & \mathbb{P}^1                                       &                        &                       
\end{tikzcd}
\end{equation}

The genera of these curves are computed by the Hurwitz formula.

\begin{lem}
    We have $g(\tC) = d(g-1)+1, \  g(C_{\sigma^p}) = p(g-1)+1,\ g(C_0) = \frac{(d-1)(g-1)}{2}, \ g(E_0) = \frac{(p-1)(g-1)}{2}$. 
\end{lem}

The following lemma is analogous to Lemma \ref{elliptic} and here one needs to use the fact that  $q>2$.
\begin{lem}\label{elliptic2}
   The covering $C_0\to E_0$ of degree $q$ is never Galois. Its Galois closure is $\tC\to E_0$. 
\end{lem}

The main idea of the proof of the (generic) injectivity of the Prym map is to reduce the problem to coverings of prime degree. Notice that the sub-diagram  of Diagram 
\eqref{small-tower} that starts from $C_{\sigma^p}$ coincides with the diagram of degree $p$ cyclic coverings of hyperelliptic curves. To shorten notation 
let $h:C\to H$ be a cyclic covering of degree $p$ and let $E_0$ be the quotient of $C$ by a lift of the hyperelliptic involution 
$\iota$ from $H$. By \cite{Ortega}, we know that $P(h)\cong JE_0\times JE_0$ with a non-product polarisation.

The following proposition can be seen as a corollary of \cite[Theorem 1.1]{NOPS}.
\begin{prop}\label{simpleness}
    Let $h:C\to H$ be an \'etale cyclic covering of prime odd degree $p$ over a hyperelliptic curve $H$ of genus $g$. Assume $h$ to be general in the moduli space $\mathcal{RH}_g[p]$. Then $\Hom(JH,P(h))=0$.
\end{prop}
\begin{proof}
Recall that $P(h)$ is (abstractly) isomorphic to $JE_0\times JE_0$. By \cite[Theorem 1.1]{NOPS} we have that for a general covering $h:C\to H$, the Jacobian $JE_0$ is simple, hence, by simplicity of $JE_0$, any proper abelian subvariety of $P(h)$ is isogenous to $JE_0$.  By assumption of 
generality, we can also assume $JH$ is simple, hence not isogenous to (a non-simple) $P(h)$. Therefore, $\Hom(JH,P(h))\neq0$ if and only if $JH$ is isogenous to $JE_0$. Now, note that $JH$ is of dimension $g$ while $JE_0$ is of dimension $\frac{(p-1)(g-1)}{2}$. An elementary computation shows that dimensions coincide (assuming $p$ is an odd prime) only if $g=2, p=5$. However, in this case, by \cite{AlbPir} the Prym map $\cP_{\cH_2}[5]$ has $1$-dimensional fibres, hence a general element in the fibre cannot have $JH$ isogenous to $JE_0$, because there are only countably many such surfaces isogenous to a given $JE_0$.
\end{proof}
We also need a generalisation of \cite[Lemma 3.1]{BB} to higher dimensions. 
\begin{lem}\label{lemboxplus}
    Let $A$ be an abelian variety, $M,N$ its subvarieties satisfying $A=M\boxplus N$ with $\Hom(M,N)=0$ and $M$ simple. Then, any abelian subvariety of $A$ of dimension $\dim(M)$ is  either equal $M$ or is contained in $N$.
\end{lem}
\begin{proof}
    Let $X\subset A$ be an abelian subvariety of dimension $\dim(X)=\dim(M)$.
    By \cite[p.125]{BL}, we have the 
    equality $e(M)e(N)id = e(N)\Nm_M + e(M)\Nm_N$, that can be restricted to $X$. If $\Nm_M(X)=0$ then 
    $e(M)e(N)X=e(M)\Nm_N(X)$ and since the image of $\Nm_N$ is $N$, we have that $X\subset N$. If $\Nm_N(X)=0$ then by the same argument, we have $X\subset M$ 
    and by equality of dimensions, we have $X=M$. 

Hence, assume $\Nm_M(X)=X_M\neq 0,\ \Nm_N(X)=X_N\neq0$. Since $M$ is simple, we have $X_M=M$ and since $\dim(X)=\dim(M)$,  $\Nm_{M}|_{X}$ is an isogeny. Then, \cite[Proposition 1.2.6]{BL} implies that there exists a map $g:M\to X$ such that $g\circ \Nm_{M}|_{X}=m_k$, where $m_k$ is multiplication by $k$ on $X$, hence $g$ is surjective. Now, $\Nm_N\circ g\neq 0\in \Hom(M,N)$ gives a contradiction.
\end{proof}

\begin{lem}\label{lemunique}
    Let $p$ be an odd prime number.
    For a general $[h:C\to H]\in \mathcal{RH}_g[p]$, 
    there is no other $h':C'\to H'$ with $C'\simeq C$ and $P(h)\simeq P(h')$. 
    In particular, the embedding $P(h)\subseteq JC$ is unique (up to an automorphism of $P(h)$).
\end{lem}
    \begin{proof}
        Fix a general covering $h:C\to H$ with its top curve $C$ and $P(h)$. Note that, by definition of the Prym variety, we have $JC=h^*JH\boxplus P(h)$ and by Proposition \ref{simpleness} we can assume $JH$ is simple and $\Hom(JH,P(h))=0$.  Assume we have another covering $h':C\to H'$ with $P(h')\cong P(h)$.
         Since $h'^*JH'$ is of dimension $\dim JH$, by 
         Lemma \ref{lemboxplus} we have  that $h'^*JH'=h^*JH$ or 
         $h'^*JH'\subseteq P(h)$. In the latter case, note that $h^*JH$ being orthogonal to $P(h)$
         would be orthogonal to $h'^*JH'$ and hence $h^*JH$ would have to be contained in $P(h')$. This is absurd, since $P(h')\cong P(h)$ and $\Hom(JH,P(h))=0$. 

        In the former case,  the equality of subvarieties $h'^*(JH')=h^*JH$, implies $P(h')=P(h)$ as subvarieties of $JC$. This shows that the embedding $P(h)\subseteq JC$ is unique up to an automorphism of $P(h)$. 
        Moreover, using Proposition \ref{group}, we have the equality $\left<\sigma\right>=\ZZ_p=\left<\sigma'\right>$, so $h=h'$ and $H=H'$.
    \end{proof}

Now, we are ready to proof main theorems of this section.
\begin{thm}\label{mainodd}
Let $d$ be an odd non-prime number. Then, the Prym map  $$\mathcal{P_H}_g[d]: \mathcal{RH}_g[d]\to\mathcal{A}_{(g-1)(d-1)}^{(1,1,\ldots, 1,d,\dots,d)}$$ is generically injective.
\end{thm} 

\begin{proof}
	Let $(P,\Theta_P)$ be an element of the image of $\mathcal{P}_{d}$. By Proposition \ref{group} we recover $\langle \sigma, -j  \rangle \subset \Aut(P,\Theta_P)$ as a subgroup isomorphic to $D_{d}$. 
    Note that $\mathbb{Z}_{d} \simeq \langle \sigma \rangle \subsetneq \langle \sigma, -j \rangle \simeq D_{d}$ is the unique subgroup of $D_{d}$ isomorphic to $\mathbb{Z}_{d}$, hence we can choose $\sigma \in \Aut(P,\Theta_P)$ to be one of the generators of $\mathbb{Z}_{d}$. 
    The elements of $D_{d} \setminus \mathbb{Z}_{d}$ are involutions of the form $-j\sigma^s$, where $s = 0,\ldots, d-1$, so let us fix one them which we denote by $-j$. Let $A := \Image(1+j)$ and  $B := \Image(1+j)(1+\sigma^p+\ldots+\sigma^{p(q-1)})$. 
    
Let $f:\tC\to H$ be an element in the preimage $\mathcal{P_H}_g[d]^{-1}(P)$. Then, by Proposition \ref{group} we have $j\in \Aut(\tC)$ is an involution with fixed points so $A=\Image (1+j)=JC_0$ is the Jacobian of the quotient curve $C_0=\tC/j$ embedded in $P$. Similarly, $B=JE_0$ by Lemma \ref{embeddingE}.

   Since $JE_0 \subsetneq JC_0$, we can use Lemma \ref{key_lemma} to get a unique covering $C_0 \ra E_0$ which recovers the covering map of degree $q$ given in Diagram \eqref{small-tower}.

    By Lemma \ref{elliptic2} the Galois closure of $C_0 \to E_0$ is precisely the covering $\tC \to E_0$ and it comes  with an automorphism $\sigma^p$ of order $q$ and an involution $j$ on $\tC$. Dividing by $\sigma^p$, we get a curve $C_{\sigma^p}$. 
    
       Denote by $g:\tC\to C_{\sigma^p}$ and $h:C_{\sigma^p}\to H$ the quotient maps, where $h$ is a quotient of $C_{\sigma^p}$ by the push-down of $\sigma\in\Aut(\tC)$. Note that $J\tC=f^*(JH)\boxplus 
       P(f)$ with $P\cong P(f)$ and $g^*JC_{\sigma^p}=g^*(h^*(JH))\boxplus g^*(P(h))$. However, in the proof of Lemma \ref{embeddingE} we have shown $g^*_{|P(h)}$ is an embedding, so $g^*JC_{\sigma^p}\cap P(f)=g^*P(h)\cong P(h)$. 

    Now, to sum up, from the Prym variety $P$ we have recovered $C_{\sigma^p}$ and $JE_0\times JE_0$ which is isomorphic to the  Prym variety of some covering $h$ of degree $p$. Hence,  assuming $P$ is general, according to Lemma \ref{lemunique}, $h$ is uniquely defined by the construction and hence $f=h\circ g$ is unique.
    Therefore, the Prym map is generically injective. 
    \end{proof}

If $d$ is not a power of an odd prime, we have a stronger result.
\begin{thm}\label{mainodd2primes}
Let $g\geq 2$ and let $d$ be an odd positive integer that has at least two different prime divisors. Then, the Prym map  $$\mathcal{P_H}_g[d]: \mathcal{RH}_g[d]\to\mathcal{A}_{(g-1)(d-1)}^{(1,1,\ldots, 1,d,\dots,d)}$$ is injective.
\end{thm} 
\begin{proof}
  Let $p,p'$ be distinct prime divisors of $d$. Following the proof of Theorem \ref{mainodd}   one obtains a covering  $C_{\sigma^p}$ and the automorphism $\sigma^p$ on 
  $\tC$. Similarly,  for $p'$ one obtains  $C_{\sigma^{p'}}$ and an  automorphism $\sigma^{p'}$. Since $p$ and $p'$ are coprime, 
  there are integers $k, l$ such that $1=kp+lp'$.
  Hence $\sigma = (\sigma^{p})^k (\sigma^{p'})^l$, 
  that is, these automorphisms generate the automorphism $\sigma$ and we get the injectivity of the Prym map without assuming that $P$ is general.
\end{proof}

\section{Application to the (global) Prym map of cyclic coverings}
Denote the moduli space of cyclic \'etale coverings of degree $d$ of genus $g$ curves by $\cR_g[d]$ and the Prym map by $$\cP_g[d]:\cR_g[d]\to\cA_{(g-1)(d-1)}^{(1,\ldots,1,d,\ldots,d)}$$

For $g=2$ every curve is hyperelliptic, so $\cP_2[d]=\cP_{\cH_2}[d]$. Merging our results with the known ones, we obtain the following theorem that describes  the fibres of Prym maps of cyclic \'etale coverings of genus 2 curves. 
\begin{thm}\label{main}
Let $d$ be a positive integer. Then, the Prym map  $$\mathcal{P}_2[d]: \mathcal{R}_2[d]\to\mathcal{A}_{d-1}^{(1,1,\ldots, 1,d)}$$ has the following  fibers depending on $d$:
\begin{itemize}
    \item if $d=2, 3$ then the generic fiber is 2-dimensional.
    \item if $d=4,5$ then the generic fiber is 1-dimensional.
    \item if $d=2q, q\geq 3$ the map is injective
 \item if $d=7$ then the map is of generic degree 10. 
 \item if $d=p^k, d>7$ then the map is generically injective. 
 \item if $d$ is odd and not a power of a prime then the map is injective. 
\end{itemize}
\end{thm} 
\begin{proof}
    The proof of the first two cases can be found in \cite{BS26} and references therein. The case $d=7$ is treated in \cite{LO16}, primes $d > 7$ for genus 2 curves are solved in \cite{NOPS} and the rest follows from Theorems \ref{maineven}, \ref{mainodd}, \ref{mainodd2primes}.
\end{proof}

\begin{rem}
    We would like to expand a little the result for $g=2, d=6$ case. It is a border case, since for $d\leq 5$ the Prym map is not finite and for $d=7$ the map is generically of degree 10. It is a little surprising that we have injectivity of the Prym map in this case. Note that by Proposition \ref{decomp} the Prym variety in this case is isogenous to $E\times F^2\times G$ for some elliptic curves $E,F,G$ and hence it is completely decomposable. Injectivity implies that the image of the Prym map is of dimension 3, so for a general triple $E,F,G$ there exists $f:\tC\to H$, such that $P(\tC/H)\sim E\times F^2\times G^2$.
\end{rem}

Now, let us consider $g>2$. We start with the following proposition. 
\begin{prop}\label{extending}
    Let $d\geq 3$ be an integer and let $A\in\Image(\cP_{\cH_g})[d]$. Then $$(\cP_g[d])^{-1}(A)=(\cP_{\cH_g}[d])^{-1}(A).$$
\end{prop}
\begin{proof}
    Let $f:\tC\to C$ be such that $\cP_g[d](f)=A$ with the deck transformation $\s\in\Aut(\tC)$. We only need to show that $C$ is hyperelliptic.
This follows from Proposition \ref{group}, where in the proof we have not used the fact that $H$ is hyperelliptic. To be precise, we have the group $D_d \simeq G \subset \Aut(A)$ intrinsically defined by the polarisation of $A$. It follows from the proof of Proposition \ref{group} that each $\psi\in G$ induces an automorphism $\tilde{\psi}$ of $J\tC$ and there is an automorphism $\tilde{\psi_0}$ of $\tC$ inducing either $\tilde{\psi}$ or $-\tilde{\psi}$. 
If $\tilde{\psi_0}$ induces $\tilde{\psi}$ then $\tilde{\psi}_0$ is the deck transformation of $f$ that is generated by $\s$ and $\psi \in \langle\sigma\rangle \subset G$. Hence if we take an involution $\psi\in G\setminus\langle\s\rangle$ we know that $-\tilde{\psi}$ comes from the automorphism $\tilde\psi_0$ of $\tC$. Since $\langle\s\rangle \simeq \mathbb{Z}_d$ is normal in $D_d$, we can push $\tilde{\psi}_0$ to a well-defined automorphism $\psi_0$ of $C$ that by construction induces $(-1)$ on $JC$, hence $\psi_0$ is the hyperelliptic involution on $C$ and we are done. 
    \end{proof}
Now we are ready to prove the following theorem.
\begin{thm}\label{finite}
    Let $d$ be a positive integer. The Prym map 
    $$\cP_3[d]:\cR_3[d]\to\cA_{2d-2}^{(1,\ldots,1,d,d)}$$
    is generically finite if and only if $d\geq 4$.
    \end{thm}
\begin{proof}
    Firstly, note that $\cP_3[d]$ is not generically finite for $d\leq3$, see \cite[Proposition 3.1]{LOarchiv}, so from now on we assume $d\geq 4$. We will use the fact that the Prym map restricted to a component containing hyperelliptic curves is generically finite. For prime numbers $d$ this has already been achieved in \cite[Theorem 3.7]{NOPS} while for composite $d\geq 6$ it follows from Theorem \ref{maineven} and Theorem \ref{mainodd}. 
    The remaining case $d=4$ can be deduced from \cite{Sha25} in the following way. For $1 \leq m \leq 2$ we define 
    $$\mathcal{RH}_3[4]_{m} := \{(H, \langle\eta\rangle) \: | \: H \in \cH_3, \eta \in JH[4]\setminus JH[2] \text{ with } \eta^2 = \mathcal{O}_H(\sum_{i=1}^{2m}(-1)^iw_i)\}/\simeq.$$ Then we have $$\mathcal{RH}_3[4] = \mathcal{RH}_3[4]_1 \cup \mathcal{RH}_3[4]_2.$$ 
    
    Let us show that the images of components under the Prym map are disjoint, i.e. $$\mathcal{P}_{3}[4](\mathcal{RH}_3[4]_{1}) \cap \mathcal{P}_{3}[4](\mathcal{RH}_3[4]_{2}) = \varnothing.$$ Let $P \in \mathcal{P}_{3}[4](\mathcal{RH}_3[4]_{m})$. Then by Proposition \ref{group} there is an intrinsic group $D_4 = \langle \sigma, -j \rangle$ acting on $P$. It follows from \cite[Theorem 5.3.3]{LaRe} that for any cover $[\tC \to H] \in (\mathcal{P}_{\mathcal{H}_3}[4])^{-1}(P)$ the isotypical decomposition of $P$ under the action of $\langle \sigma, j \rangle$ includes the term $h^*JC_{\sigma^2, j}$ (which corresponds to the character $\chi$ given by $\chi(\sigma^ij) = (-1)^i$), where $h: \tC \to \tC / \langle \sigma^2, j \rangle =: C_{\sigma^2, j}$. In \cite{Sha25} it was proved that $g(C_{\sigma^2, j}) = m-1$, hence Prym varieties corresponding to different components are non-isomorphic. Now, \cite[Theorem 5.2]{Sha25} states that $\mathcal{P}_{\mathcal{H}_3}[4]$ is in fact injective on $\mathcal{RH}_3[4]_{1}$. 
    
    To finish the proof of generic finiteness of the Prym map it is enough to show that the dimension of the image of the Prym map equals the dimension of the domain, which is 6. By contradiction, assume $\dim(\Image\cP)\leq 5$. By the first part of the proof, we know that  $\dim(\cP(\cR\cH_3))=\dim(\cR\cH_3)=5$, hence $\dim(\Image\cP)=5$ and $\cP(\cR\cH_3)$ is a dense open subset of the image of the Prym map. But now, Proposition \ref{extending} says that generically (taking this open subset) $\cP$ is injective, which is absurd. Hence, $\dim(\Image\cP)=6$ and the map is generically finite.     
\end{proof}
Note that the Prym map is not proper, so for $g>3$, even though Proposition \ref{extending} shows injectivity on the locus of hyperelliptic curves that is of codimension at least 2, we cannot deduce generic finiteness of the Prym map. 

Finally, Proposition \ref{extending} provides an evidence that the generic degree of the Prym map should equal 1 not only for $g=2$ but also for higher genera. Hence, we state the following conjecture.
\begin{conjecture}\label{conjecture}
 For $g\geq 3$ and $d\geq 4$ the generic Prym Torelli holds, i.e. the Prym map $\cP_g[d]$ is injective.
\end{conjecture}
In principle, if the degree $d$ is a prime number that does not satisfy the mild conditions from \cite{NOPS} then we do not have an evidence for the injectivity. However, we believe it is more a weakness of the proof rather than an actual constrain.

We would also like to mention that in \cite{Ago} it is proved that for $g=2$ the differential of the Prym map is non-injective at the locus of bielliptic curves. In particular, in those degrees for which the Prym map is known to be injective, Agostini's result shows that it need not be an embedding. Note that for higher genera we do not have such a description.

 \section{Statements}
 The authors state that there is no conflict of interest.
 The manuscript has no associated data.

\subsection*{Statement on AI use}
All mathematical results, arguments, and proofs in this article are the work of the authors. The authors acknowledge the use of LLMs to check the final version of the manuscript for possible mathematical errors and typographical mistakes.

\bibliographystyle{alpha}
\bibliography{bibl}

\end{document}